\documentclass{amsart}

\usepackage{enumerate}
\usepackage{amsmath, amscd}
\usepackage{epsfig}
\usepackage{graphics}
\usepackage[normalem]{ulem}
\usepackage{color}
\usepackage{hyperref}

\newtheorem{thm}{Theorem}
\newtheorem{main theorem}[thm]{Main Theorem}
\newtheorem{corollary}[thm]{Corollary}

\newtheorem{lemma}[thm]{Lemma}

\newtheorem{conjecture}[thm]{Conjecture}
\newtheorem{problem}[thm]{Problem}
\theoremstyle{definition}

\newtheorem{remark}[thm]{Remark}

\newcommand{\bea}{\begin{eqnarray*}}
\newcommand{\eea}{\end{eqnarray*}}

\newcommand{\be}{\begin{equation}}
\newcommand{\ee}{\end{equation}}
\begin{document}

\title{A survey on non-autonomous basins in several complex variables}

\author[Abbondandolo, Arosio, Forn{\ae}ss, Majer, Peters, Raissy, Vivas]{Alberto Abbondandolo, Leandro Arosio, John Erik Forn{\ae}ss, \\
Pietro Majer, Han Peters, Jasmin Raissy, Liz Vivas}

\address{Fakult\"at f\"ur Mathematik\\
Ruhr-Universit\"at Bochum\\
Germany}
\email{Alberto.Abbondandolo@rub.de}

\address{Dipartimento di Matematica\\
Universit\`a di Roma “Tor Vergata”\\
Italy}
\email{arosio@mat.uniroma2.it}

\address{Department for Mathematical Sciences\\
Norwegian University of Science and Technology\\
Trondheim, Norway}
\email{john.fornass@math.ntnu.no}

\address{Dipartimento di Matematica\\
Universit\`a di Pisa\\
Italy}
\email{majer@dm.unipi.it}

\address{KdV Institute for Mathematics\\
University of Amsterdam\\
The Netherlands}
\email{h.peters@uva.nl}

\address{Institut de Math\'ematiques de Toulouse\\
Universit\'e Paul Sabatier\\
France}
\email{jraissy@math.univ-toulouse.fr }

\address{Universidad Federal Fluminense\\
Rio de Janeiro\\
Brasil}
\email{lizvivas@impa.br}

\begin{abstract}
Consider a holomorphic automorphism which acts hyperbolically on some invariant compact set. Then for every point in the compact set there exists a stable manifold, which is a complex manifold diffeomorphic to real Euclidean space. If the point is fixed, then the stable manifold is even biholomorphic to complex Euclidean space. In fact, it is known that the stable manifold of a generic point is biholomorphic to Euclidean space, and it has been conjectured that this holds for every point.

In this article we survey the history of this problem, addressing both known results and the techniques used to obtain those results. Moreover, we present a list of seemingly simpler open problems and prove several new results, all pointing towards a positive answer to the conjecture discussed above.
\end{abstract}

\maketitle

\section{Introduction}

Let $X$ be a complex manifold equipped with a Riemannian metric, and let $f:X \rightarrow X$ be an automorphism which acts hyperbolically on some invariant compact subset $K \subset X$. Let $p \in K$ and write $\Sigma_f^s(p)$ for the stable manifold of $f$ through $p$. $\Sigma_f^s(p)$ is a complex manifold, say of complex dimension $m$. In the special case where $p$ is a fixed point it is known that $\Sigma_f^s(p)$ is biholomorphically equivalent to $\mathbb C^m$. It was conjectured by Bedford \cite{Bedford} that this equivalence holds for any $p \in K$.

\begin{conjecture}[Bedford] \label{conj:stable}
The stable manifold $\Sigma_f^s(p)$ is always equivalent to $\mathbb C^m$.
\end{conjecture}

The usual approach towards this problem is to translate it to the following stronger conjecture regarding non-autonomous basins. Let $f_0, f_1, \ldots$ be a sequence of automorphisms of $\mathbb{C}^m$ satisfying
\begin{equation}\label{eq:uniform}
C\|z\| \le \|f_n(z)\| \le D\|z\|,
\end{equation}
for all $n \in \mathbb{N}$ and all $z$ lying in the unit ball $\mathbb{B}$, and where the constants $1>D>C>0$ are independent of $n$. Define the basin of attraction of the sequence $(f_n)$ by
\begin{equation}
\Omega = \Omega_{(f_n)} = \{z \in \mathbb{C}^k \mid f_n \circ \cdots \circ f_0(z) \rightarrow 0\}.
\end{equation}

\begin{conjecture} \label{conj:main}
The basin $\Omega$ is always biholomorphic to $\mathbb{C}^m$.
\end{conjecture}

Here we will present a survey on the history of the two conjectures mentioned above. We will raise a number of other, possibly weaker, open questions, and present several new results, all pointing towards a positive answer to the above conjectures.

In section 2 we give an overview of related known results, and in section 3 we present some of the techniques that were used to prove those results. In section 4 we mention a number of open problems, and in sections 4 and 5 we prove our new results.

The fifth author was supported by a SP3-People Marie Curie Actionsgrant in the project Complex Dynamics (FP7-PEOPLE-2009-RG, 248443).

\section{A short history}

Let $f$ be an automorphism of a complex manifold $X$ of dimension $m$, and let $p \in X$ be an attracting fixed point. Define the basin of attraction by
$$
\Omega = \{ z \in X \; \mid \; f^n(z) \rightarrow p\}.
$$
The following result was proved independently by Sternberg \cite{Sternberg} and Rosay-Rudin \cite{RR}.

\begin{thm}[Sternberg, Rosay-Rudin]
The basin $\Omega$ is biholomorphic to $\mathbb C^m$.
\end{thm}

If $p$ is not attracting but hyperbolic then if one considers the restriction of $f$ to the stable manifold, the restriction is an automorphism of $\Sigma^s_f(p)$ with an attracting fixed point at $p$. Moreover, its basin of attraction is equal to the entire stable manifold, which is therefore equivalent to $\mathbb C^m$. This naturally raised Conjecture \ref{conj:stable}. Equivalence to $\mathbb C^m$ of \emph{generic} stable manifolds was proved by Jonsson and Varolin in \cite{JV}.

\begin{thm}[Jonsson-Varolin]\label{thm:JV}
Let $X$, $f$ and $K$ be as in Conjecture \ref{conj:stable}. For a generic point $p \in K$ the stable manifold through $p$ is equivalent to $\mathbb C^m$.
\end{thm}

Here generic refers to a subset of $K$ which has full measure with respect to any invariant probability measure on $K$. In fact Jonsson and Varolin showed that Conjecture \ref{conj:stable} holds for so-called Oseledec points. Results of Jonsson and Varolin were extended by Berteloot, Dupont and Molino in \cite{BDM}. More recently the following was shown in \cite{AAM}.

\begin{thm}[Abate-Abbondandolo-Majer]\label{thm:Lyapunov}
The existence of Lyapunov exponents is enough to guarantee $\Sigma_f^s(p) \cong \mathbb C^m$.
\end{thm}

Shortly after the positive result of Jonsson and Varolin, Forn{\ae}ss proved the following negative result, which led many people to believe that Conjectures \ref{conj:stable} and \ref{conj:main} must be false. Consider a sequence of maps $(f_n)_{n\ge 0}$ given by
$$
f_n: (z,w) \rightarrow (z^2 + a_n w, a_n z),
$$
where $|a_0| < 1$ and $|a_{n+1}| \le |a_n|^2$.

\begin{thm}[Forn{\ae}ss]\label{thm:short}
The basin $\Omega_{(f_n)}$ is not biholomorphic to $\mathbb C^2$. Indeed there exist a bounded pluri-subharmonic function on $\Omega_{(f_n)}$ which is not constant.
\end{thm}

We note that this result does not give a counterexample to Conjecture \ref{conj:main}, as the sequence $(f_n)$ violates the condition
$$
C\|z\| \le \|f_n(z)\| \le D \|z\|
$$
on any uniform neighborhood of the origin. In Theorem \ref{thm:short} the rate of contraction is not uniformly bounded by below. In fact, we will show (see Lemma \ref{lemma:psh} in section 4) that under the assumption of uniform contraction (Equation \eqref{eq:uniform}) all bounded pluri-subharmonic functions are constant.

If the bounds $C$ and $D$ satisfy $D^2 < C$ then it was shown by Wold in \cite{Wold} that the basin of attraction is biholomorphic to $\mathbb C^m$. This result was generalized by Sabiini \cite{Sabiini} to the following statement.

\begin{thm}\label{thm:wold}
Let $(f_n)$ be a sequence of automorphisms which satisfies the conditions in Conjecture \ref{conj:main}, and suppose that $D^k < C$ for some $k \in \mathbb N$. Suppose further that the maps $f_n$ all have order of contact $k$. Then the basin of attraction $\Omega_{(f_n)}$ is biholomorphic to $\mathbb C^m$.
\end{thm}

Here a map $f$ has order of contact $k$ if $f(z) = l(z) + O(\|z\|^k)$, where $l(z)$ is linear.
 
\medskip

More recently the following was shown in \cite{AM}.

\begin{thm}[Abbondandolo-Majer]\label{thm:AM}
There exist a constant $\epsilon  = \epsilon(m, k) > 0$ so that the following holds. For any sequence $(f_n)$ with order of contact $k$, which satisfies the conditions in Conjecture \ref{conj:main} with
$$
D^{k+\epsilon} < C,
$$
the basin $\Omega_{(f_n)}$ is biholomorphic to $\mathbb C^m$.
\end{thm}

While this result only seems marginally stronger (indeed, $\epsilon$-stronger), it is in fact a much deeper result. If $D^k < C$ one can ignore all but the linear terms, see Lemma \ref{lemma:conjugation} in the next section. But if $D^k \ge C$ then one has to deal with the terms of order $k$, which is a major difficulty. In fact, looking at more classical results in local complex dynamics, the major difficulty in describing the behavior near a fixed point usually lies in controlling the lowest order terms which are not trivial. Theorem \ref{thm:AM} and Theorem \ref{thm:main} below may therefore be important steps towards a complete understanding of Conjecture \ref{conj:main}.

Now that techniques have been found to deal with these terms of degree $k$, it is natural to ask whether the condition $D^{k+\epsilon} < C$ can be pushed to the condition $D^{k+1} < C$, since as long as $D^{k+1} < C$ is satisfied one can ignore terms of degree strictly greater than $k$. We will show here that indeed we can weaken the requirement to $D^{k+1} < C$, at least in $2$ complex dimensions and under the additional assumption that the linear parts of all the maps $f_n$ are diagonal. Whether diagonality is a serious extra assumption or merely simplifies the computations remains to be seen.

\begin{thm}\label{thm:main}
Let $(f_n)$ be a sequence of automorphisms of $\mathbb C^2$ which satisfies the conditions in Conjecture \ref{conj:main}, and suppose that the maps $f_n$ all have order of contact $k$. Assume further that the linear part of each map $f_n$ is diagonal and that 
$$
D^{k+1} < C.
$$
Then the basin of attraction $\Omega_{(f_n)}$ is biholomorphic to $\mathbb C^2$.
\end{thm}

Theorem \ref{thm:main} will be proved in section (5). The following result was proved in \cite{PW}.

\begin{thm}[Peters-Wold]\label{thm:repeat}
Let $(f_j)$ be a sequence of automorphism of $\mathbb C^m$, each with an attracting fixed point at the origin. Then there exist a sequence of integers $(n_j)$ so that the attracting basin of the sequence $(f_j^{n_j})$ is equivalent to $\mathbb C^m$.
\end{thm}

The idea that appeared in Theorem \ref{thm:repeat} will also be important in the proof of Theorem \ref{thm:main}: if long stretches in the sequence $(f_n)$ behave similarly, then the basin will be biholomorphic to $\mathbb C^m$. What is meant by similarly will be made clear in later sections.

The next result, from \cite{Peters}, will be important to us not so much for the statement itself but for the ideas used in the proof. See also the more technical Lemma \ref{lemma:dominant}, which will be discussed in the next section.

\begin{thm}[Peters]\label{thm:perturbations}
Let $F$ be an automorphism of $\mathbb C^m$. Then there exist an $\epsilon > 0$ so that for any sequence $(f_n)$ of automorphisms of $\mathbb C^m$ which all fix the origin and satisfy $\|f_n - F\|_{\mathbb B} < \epsilon$ for all $n$, one has that $\Omega_{(f_n)} \cong \mathbb C^m$.
\end{thm}

\subsection{Applications}

The results described above have been used to prove a number of results which are at first sight unrelated. The first example of such an application is of course the classical result of Fatou and Bieberbach which states that there exists a proper subdomain of $\mathbb C^2$ which is biholomorphic to $\mathbb C^2$. Indeed, it is not hard to find an automorphism of $\mathbb C^2$ with an attracting fixed point, but whose basin of attraction is not equal to the entire $\mathbb C^2$. Proper subdomains of $\mathbb C^2$ that are equivalent to $\mathbb C^2$ are now called \emph{Fatou-Bieberbach} domains.

It should be no surprise that for the construction of Fatou-Bieberbach domains with specific properties, it is useful to work with non-autonomous basins. Working with sequences of maps gives much more freedom than working with a single automorphism. Using Theorem \ref{thm:wold} it is fairly easy to construct a sequence of automorphisms satisfying various global properties, while making sure that the attracting basin is equivalent to $\mathbb C^2$. We give two examples of results that have been obtained in this way.

\begin{thm}[Wold]
There exist a Fatou-Bieberbach domain $\Omega \subset \mathbb C^2$ which is dense in $\mathbb C^2$.
\end{thm}

\begin{thm}[Peters-Wold]
There exists a Fatou-Bieberbach domain $\Omega$ in $\mathbb C^2$ whose boundary has Hausdorff dimension $4$, and even positive $4$-dimensional Lebesgue measure.
\end{thm}

Here we note that Fatou-Bieberbach domains with Hausdorff dimension equal to to any $h \in (3,4)$ were constructed by Wolf \cite{Wolf} using autonomous attracting basins. Hausdorff dimension $3$ (and in fact $C^\infty$-boundary) was obtained earlier by Stens{\o}nes in \cite{Stensones}, who also used an iterative procedure involving a sequence of automorphisms of $\mathbb C^2$.

The last application we would like to mention is the Loewner partial differential equation. The link with non-autonomous attracting basins was made in \cite{Arosio}, where Arosio used a construction due to Forn{\ae}ss and Stens{\o}nes (see Theorem \ref{thm:FS} below) to prove the existence of solutions to the Loewner PDE. See also \cite{ABW} for the relationship between the Loewner PDE and non-autonomous attracting basins.

\section{Tools and techiniques}

\subsection{The autonomous case}

Essentially the only available method for proving that a domain $\Omega$ is equivalent to $\mathbb C^m$, is by actually constructing the biholomorphic map from $\Omega$ to $\mathbb C^m$. Let us first review how this is done for autonomous basins, before we look at how this proof can be adapted to the non-autonomous setting. We follow the proof in the appendix of \cite{RR}, but see also the survey \cite{Berteloot} written by Berteloot. Given an automorphism $F$ of $\mathbb C^m$ with an attracting fixed point at the origin, one can, for any $k \in \mathbb N$, find polynomial maps $X_k$ of the form $X_k = \mathrm{Id} + h.o.t.$, biholomorphic in a neighborhood of the origin, and a \emph{lower triangular polynomial map} $G$ so that
\begin{equation}\label{eq:conjugate}
X_k \circ F = G \circ X_k + O(\|z\|^{k+1}).
\end{equation}
Here a polynomial map $G = (G_1, \ldots , G_m)$ is called \emph{lower triangular} if $G_i = \lambda_i z_i + H_i(z_1, \ldots z_{i-1})$ for $i= 1, \ldots , m$. Now consider the sequence of holomorphic maps  from $\Omega_F$ to $\mathbb C^m$ given by
$$
\Phi_n = G^{-n} \circ X_k\circ  F^n.
$$
If $k$ is chosen sufficiently large then it follows from Equation \eqref{eq:conjugate} that the maps $\Phi_n$ converge, uniformly on compact subsets of $\Omega_F$, to a biholomorphic map from $\Omega_F$ to $\mathbb C^m$. Important here is that the lower triangular polynomial maps in many ways behave as linear maps. For example, it follows immediately by induction on $m$ that the degrees of the iterates $G^n$ are uniformly bounded. One can also easily see that the basin of attraction of a lower triangular map with an attracting fixed point at the origin is always equal to the whole set $\mathbb C^m$. As we will see below, the fact that $F$ and $G$ are conjugate as jets of sufficiently high degree implies that their basins are equivalent, and hence the basin of the map $F$ is equivalent to $\mathbb C^m$.

\subsection{Non-autonomous Conjugation}

In the non-autonomous setting it is very rare that a single change of coordinates simplifies the sequence of maps. Instead we use a sequence of coordinate changes.

\begin{lemma}\label{lemma:conjugation}
Let $(f_n)$ be a sequence of automorphisms that satisfies the hypotheses of Conjecture \ref{conj:main}, and suppose that there exist uniformly bounded sequences $(g_n)$ and $(h_n)$, with $h_n = \mathrm{Id} + h.o.t.$, such that the diagram
\begin{equation}\label{diagram}
\begin{CD}
\mathbb{C}^m @>f_0>> \mathbb{C}^m @>f_1>> \mathbb{C}^m @>f_2>> \cdots\\
@VVh_0V @VVh_1V @VVh_2V\\
\mathbb{C}^m @>g_0>> \mathbb{C}^m @>g_1>> \mathbb{C}^m @>g_2>> \cdots
\end{CD}
\end{equation}
commutes as germs of order $k$. Then $\Omega_{(f_n)} \cong \Omega_{(g_n)}$.
\end{lemma}

If the maps $(g_n)$ are all lower triangular polynomials then one still has that  the basin of the sequence $(g_n)$ is equal to $\mathbb C^m$. This simple fact was used in \cite{Peters} to prove the following lemma, which for simplcity we state in the case $m = 2$.

\begin{lemma}\label{lemma:dominant}
Let $(f_n)$ be a sequence of automorphisms satisfying the conditions in Conjecture \ref{conj:main}, and suppose that the linear part of each map $f_n$ is of the form
$$
(z, w) \mapsto (a_n z, b_n w + c_n z),
$$
with $|b_n|^2 < \xi |a_n|$ for some uniform constant $\xi < 1$. Then we can find bounded sequences $(g_n)$ and $(h_n)$ as in Lemma \ref{lemma:conjugation}. Moreover, the maps $g_n$ can be chosen to be lower triangular polynomials, and hence $\Omega_{(f_n)} \cong \mathbb C^2$.
\end{lemma}

As was pointed out in \cite{Peters}, we can always find a non-autonomous change of coordinates so that the linear parts of the maps $f_n$ all become lower triangular. Let us explain the technique in the $2$-dimensional setting, where a matrix is lower diagonal if and only if $[0,1]$ is an eigen vector. Using $QR$-factorization we can find, for any vector $v_0$, a sequence of unitary matrices $(U_n)$ so that
$$
U_{n+1} \cdot D(f_n \circ \cdots \circ f_0)(0) v_0 = \lambda_n \cdot [0,1],
$$
with $\lambda_n \in \mathbb C$. Then if we defines $g_n = U_{n+1} \circ f_n \circ U_n^{-1}$, we obtains a new sequence $(g_n)$ whose basin is equivalent (by the biholomorphic map $U_0$) to the basin of the sequence $(f_n)$. Notice that the linear parts of all the maps $(g_n)$ are lower triangular. In this construction we are free to choose the initial tangent vector $v_0$.

But while we may always assume that the linear parts are lower triangular, the condition $|b_n|^2 < \xi |a_n|$ in Lemma \ref{lemma:dominant} is a strong assumption. In particular there is no reason to think that one can change coordinates to obtain a sequence of lower triangular polynomial maps. Instead we could aim for obtaining lower triangular polynomials on sufficiently large time-intervals. Let us be more precise. Suppose that we have found unitary matrices $U_{0,0}, \ldots U_{0,p_1}$ so that the maps $g_n = U_{0,n+1} \circ f_n \circ U_{0, n}^{-1}$ are lower triangular polynomials for $n = 0, \ldots p_1$. Suppose further that we can find a strictly increasing sequence $p_2, p_3, \ldots$ and for each $j$ unitary matrices $U_{j, p_j}, \ldots U_{j, p_{j+1}}$ so that the maps $g_n = U_{j,n+1} \circ f_n \circ U_{j, n}^{-1}$ are lower triangular. Then the basin of the sequence $(f_n)$ is equal to the basin of the sequence
\begin{equation} \label{eq:AM}
U_{0,0}, g_0, g_1, \ldots, g_{p_1-1}, U_{0,p_1}^{-1}, U_{1,p_1}, g_{p_1}, \ldots.
\end{equation}
In the spirit of Theorem \ref{thm:repeat} one would expect that the basin of this new sequence is equal to $\mathbb C^2$ as long as the sequence $(p_j)$ is sparse enough. Indeed this is the case, as follows from the following Lemma, proved by Abbondandolo and Majer in \cite{AM}.

\begin{lemma}\label{lemma:sparse}
Suppose that the maps $g_n$ in the sequence given in Equation \eqref{eq:AM} are lower triangular polynomials of degree $k$, and that
\begin{equation}\label{eq:infinitesum}
\sum \frac{p_{j+1} - p_j}{k^j} = +\infty.
\end{equation}
Then the basin of the sequence in Equation \ref{eq:AM} is equal to $\mathbb C^2$
\end{lemma}

The proof of Theorem \ref{thm:AM} from \cite{AM} can now be sketched as follows. Define $p_{j+1} = k^j + p_j$, and on each interval find a tangent vector $v_j$ which is contracted most rapidly by the maps $Df_{p_{j+1}, p_j}$. Next find the non-autonomous change of coordinates by unitary matrices so that the maps $g_n$ as defined above all have lower triangular linear part. Then on each interval $I_j = [p_j, p_{j+1}]$ the maps $g_n$ satify the conditions of Lemma \ref{lemma:dominant} ``on average''. This is enough to find another non-autonomous change of coordinates after which the maps $g_n$ are lower triangular polynomial maps on each interval $I_j$. Then it follows from Lemma \ref{lemma:sparse} that the basin of the sequence $(f_n)$ is equivalent to $\mathbb C^2$.

Now we arrive at one of the main points presented in this article. In the argument of Abbondandolo and Majer the intervals $[p_j, p_{j+1}]$ were chosen without taking the maps $(f_n)$ into consideration, it was sufficient to make a simple choice so that Equation \eqref{eq:infinitesum} is satisfied. In the last section of this article we will show that we can obtain stronger results if we instead let the intervals $[p_j, p_{j+1}]$ depend on the maps $(f_n)$, or to be more precise, on the linear parts of the maps $(f_n)$.

\subsection{Abstract Basins}

Let us discuss a construction due to Forn{\ae}ss and Stens{\o}nes \cite{FS}. Let $(f_n)$ now be a sequence of biholomorphic maps from the unit ball $\mathbb B$ into $\mathbb B$, satisfying
$$
C \|z\| \le \|f_n(z)\| \le D\|z\|
$$
for some uniform $1>D>C>0$ as usual. We define the \emph{abstract basin of attraction} of the sequence $(f_n)$ as follows. Consider all sequences of the form
$$
(x_k, x_{k+1}, \ldots), \; \; \mathrm{with} \; \; x_{n+1} = f_n(x_n) \; \; \mathrm{for} \; \mathrm{all} \; n \ge k.
$$
We say that
$$
(x_k, x_{k+1}, \ldots) \sim (y_l, y_{l+1}, \ldots)
$$
if there exists a $j \ge \max(k,l)$ such that $x_j = y_j$. This gives an equivalence relation $\sim$, and we define
$$
\Omega_{(f_n)} = \{(x_k, x_{k+1}, \ldots) \mid f_n(x_n) = x_{n+1} \}/\sim
$$
We refer to $\Omega_{(f_n)}$ as the \emph{abstract basin of attraction}, sometimes also called the \emph{tail space}. We have now used the notation $\Omega_{(f_n)}$ for both abstract and non-autonomous basins, but thanks to the following lemma this will not cause any problems.

\begin{lemma}\label{lemma:abstract}
Let $(f_n)$ be a sequence of automorphisms of $\mathbb C^m$ which satisfy the conditions in Conjecture \ref{conj:main}. Then the basin of attraction of the sequence $(f_n)$ is equivalent to the abstract basin of the sequence $(f_n|_{\mathbb B})$.
\end{lemma}

Hence from now on we allow ourselves to be careless and write $\Omega_{(f_n)}$ for both kinds of attracting basins. Abstract basins were used by Forn{\ae}ss and Stens{\o}nes to prove the following.

\begin{thm}[Forn{\ae}ss-Stens{\o}nes]\label{thm:FS}
Let $f$ and $p$ be as in Conjecture \ref{conj:stable}. Then $\Sigma^s_f(p)$ is equivalent to a domain in $\mathbb C^m$.
\end{thm}

\begin{remark} Working with abstract basins can be very convenient. For example, Lemma \ref{lemma:conjugation} also holds for abstract basins which, in conjunction with Lemma \ref{lemma:abstract}, implies that in Diagram \ref{diagram} we do not need to worry about whether the maps $h_n$ and $g_n$ are globally defined automorphisms. From the fact that the sequences $(g_n)$ and $(h_n)$ are uniformly bounded it follows that their restrictions to some uniform neighborhood of the origin are biholomorphisms, which is all that is needed.
\end{remark}

\section{Open problems}

In this section we consider a number of open problems which might be more accessible than Conjectures \ref{conj:stable} and \ref{conj:main}. As we noted in the previous section, essentially the only available method for proving that a manifold is equivalent to $\mathbb C^m$ is by constructing an explicit biholomorphism. It would therefore be very useful to find general conditions which imply that a given domain is equivalent to Euclidean space.

\begin{problem}
Let $\Omega$ be an $m$-dimensional complex manifold. Describe conditions which are sufficient to conclude that $\Omega \cong \mathbb C^2$.
\end{problem}

One possible example of such conditions is given in the following Conjecture of Yau \cite{Yau}.

\begin{problem}[Yau's uniformization conjecture]
A complete noncompact K\"ähler manifold with positive holomorphic bisectional curvature is biholomorphic to $\mathbb C^m$.
\end{problem}

The relation of Yau's Uniformization Conjecture to Conjectures \ref{conj:stable} and \ref{conj:main} may be even stronger than it seems. For example, the concept of slowly varying maps, which were used by Jonsson and Varolin in \cite{JV} to prove Theorem \ref{thm:JV}, was also used by Chau and Tam to study Yau's Uniformization Conjecture, see for example \cite{CT}.

\medskip

Let us go back to non-autonomous basins of attraction. We present several seemingly simpler situations in which it is not known whether the basin is equivalent to Euclidean space.

\begin{problem}
Let $F$ and $G$ be automorphisms of $\mathbb C^m$, both having an attracting fixed point in the origin. Let $(f_n)$ be a sequence in which each map $f_n$ is equal to either $F$ or $G$. Is $\Omega_{(f_n)}$ equivalent to $\mathbb C^m$?
\end{problem}

The above problem is open even when $F$ and $G$ are explicit and quite simple maps, such as
$$
F(z,w) = (\frac{1}{2} z + w^2, \frac{1}{9} w), \; \; \mathrm{and} \; \; G(z,w) = (\frac{1}{9} z, \frac{1}{2} w + z^2).
$$
Notice that $(\frac{1}{2})^3 > \frac{1}{9}$, hence Theorem \ref{thm:main} does not apply. As we noted before, it is not clear whether requiring that the maps $f_n$ have diagonal linear part is a significant simplification. In the case $D^{k+1} < C$ discussed in the last section the diagonality assumption significantly reduces the complexity of our computations. It would therefore be worthwhile to study the general problem under the same assumption.

\begin{problem}
Let $(f_n)$ be a sequence of automorphisms of $\mathbb C^2$ satisfying the conditions of Conjecture \ref{conj:main}, and suppose that the maps $(f_n)$ all have diagonal linear part. Does it follow that $\Omega_{(f_n)} \cong \mathbb C^2$?
\end{problem}

Instead of trying to prove that $\Omega_{(f_n)}$ is biholomorphic to Euclidean space, one could take a step back and study other (and perhaps weaker) properties of the domains $\Omega_{(f_n)}$.

\begin{problem}
What properties must a non-autonomous basin $\Omega_{(f_n)}$ satisfy?
\end{problem}

For example, we know that $\Omega_{(f_n)}$ is an increasing union of balls, and that its Kobayashi metric vanishes identically.  The following was proved by Wold in \cite{Wold}.

\begin{thm}[Wold]
Every non-autonomous basin $\Omega_{(f_n)} \subset \mathbb C^m$ is Runge. Conversely, each Runge Fatou-Bieberbach domain can be written as the basin of a sequence of automorphisms.
\end{thm}

We note that in \cite{Wold2} Wold proved that there exist Fatou-Bieberbach domains which are not Runge, and therefore cannot be written as a non-autonomous basin.

\medskip

Let us recall the Short $\mathbb C^2$-example that was constructed by Forn{\ae}ss in \cite{Fornaess}. This domain was shown not to be equivalent to $\mathbb C^2$ by constructing a bounded non-constant plurisubharmonic function on $\Omega$. We now prove the following.

\begin{lemma}\label{lemma:psh}
Let $(f_n)$ as in Conjecture \ref{conj:main}. Then there is no non-constant bounded pluri-subharmonic function on $\Omega_{(f_n)}$.
\end{lemma}
\begin{proof}
Assume that there does exist a bounded non-constant pluri-subharmonic function $\rho$ on $\Omega_{(f_n)}$. Define
$$
B_n = (f_{n-1} \circ \cdots \circ f_0)^{-1} \mathbb B.
$$
We can assume that $\rho=-1$ on $B_0 = \mathbb B$, and $\sup \rho=0$.
Consider $\rho$ restricted to $B_n$. We get a pluri-subharmonic function $\rho_n=
\rho \circ (f_{n-1} \circ \cdots \circ f_0)^{-1}$ on the unit ball $\mathbb B$. Let
$$
\sigma_n(z)=\sup \{\rho_n(w); \|w\|=\|z\|\}.
$$
Then $\sigma_n(z)= -1$ if $\|z\|<C^n.$
Since $\sigma_n$ is logarithmically convex we have that
$$
\sigma_n(z) \leq \frac{-1}{n\log C} \log |z|
$$
on the unit ball. Next, fix $k$ and assume that $n>k$. On $B_k$ we have that $|f^n(z)| \leq D^{n-k}$.
Hence
$$
\sigma_n(z)\leq \frac{-1}{n\log C} (n-k) \log D,
$$
and therefore also
$$
\rho(z)\leq  \frac{-1}{n\log C} (n-k) \log D
$$
on $B_k.$ Since this is true for all $n>k$ we have that
$$
\rho(z)\leq \frac{-\log D}{\log C}<0
$$
on $B_k$. Since this also holds for all $k$ and $\Omega_{(f_n)} = \bigcup B_k$ we cannot have
$\sup \rho=0.$
\end{proof}

Another natural question is the following.

\begin{problem}
Does $\Omega_{(f_n)}$ contain embedded or immersed complex lines? In particular, does there exist a holomorphic line through every point, and in every direction?
\end{problem}

We note that in the Short $\mathbb{C}^2$ example of Theorem \ref{thm:short}, the complex lines lying in $\Omega$ are severely restricted, since they must lie within level-sets of the pluri-subharmonic function. In particular one can show that for some points $p \in \Omega$ and most tangent vectors $v$ at $p$, there exists no holomorphic map $\phi : \mathbb C \rightarrow \Omega$ with $\phi(0) = p$ and $\phi^\prime(0) = v$. We will proceed to show that this is not the case in our setting, see Corollary \ref{cor:lines} below.

In what follows we will write $\Delta(r) \subset \mathbb C$ for the disk or radius $r$, and $\Delta$ for the unit disk.

\begin{lemma}\label{lemma:JE1}
Fix constants $R>1, c<1, r<1.$ Then there exist $L>0, \delta>0$ so that for every $\epsilon>0$ and all large enough $N$ we have uniformly for any
analytic function $g(z)=\sum a_n z^n:\Delta\rightarrow \Delta$ that
\begin{align}
\label{eq:i} \left|\sum_{n\leq LN} a_nz^n\right| & \leq R^N & \mbox{if} & \;|z|< 1+\delta, \; \; \; \mathrm{and}\\
\label{eq:ii} \left|\sum_{n>LN} a_n z^n\right| & \leq c^N \epsilon & \mbox{if} & \; |z|< r.
\end{align}
\end{lemma}
\begin{proof}
From the Cauchy Estimates we know that $|a_n|\leq 1$ for all $n$.  Hence the estimates \eqref{eq:i} and \eqref{eq:ii}
follow if
\begin{align}
\label{eq:i'} \frac{(1+\delta)^{LN+1}-1}{\delta} & < R^N, \; \; \mathrm{and}\\
\label{eq:ii'} \frac{r^{LN}}{1-r} & < c^N\epsilon.
\end{align}

Next choose first $L$ so that $r^L<c$ and then choose $\delta$ so that
$(1+\delta)^L<R.$ Then both \eqref{eq:i'} and \eqref{eq:ii'} hold for all large enough $N$ independent of $g$.
\end{proof}

We consider the basin $\Omega_{(f_n)}$ of the sequence of automorphisms $(f_j)$ with
$C\|z\|\leq \| f_j(z)\| \leq D \|z\|$ on the unit ball, $0<C<D<1$. Without lack of generality we can assume that for every $z,w$ in the unit ball,
$\|f_j(z)-f_j(w)\|\geq C\|z-w\|.$ For the rest of this section we will write $f^n$ for the composition of the first $n$ maps, i.e.
$$
f^n = f_{n-1} \circ \cdots \circ f_0.
$$
Similarly we write $f^{-n}$ for the inverse of $f^n$.

\begin{lemma}\label{lemma:JE3}
Let $\eta<1/2.$ Suppose that $z\in f_j(\mathbb B(0,1/2))$ and suppose that
$\|w-z\|< \eta C.$ Then $\|f_j^{-1}(w)-f_j^{-1}(z)\|< \eta$. In particular,
$f_j^{-1}(w)\in \mathbb B(0,1).$
\end{lemma}

\begin{proof}
The statement follows from the open mapping theorem, since the distance from $f_j(\partial \mathbb B(f_j^{-1}(z),\eta))$ to $z$ is at least $\eta C$.
\end{proof}

\begin{lemma}\label{lemma:JE4}
Let $\eta<1/2$. Suppose that $z\in f^N(\mathbb B(0,1/2))$ and suppose that
$\|w-z\|< \eta C^N.$ Then $\|(f^N)^{-1}(w)-(f^N)^{-1}(z)\|< \eta$. In particular,
$(f^N)^{-1}(w)\in \mathbb B(0,1).$
\end{lemma}

\begin{proof}
The statement follows from applying the previous lemma $N$ times.
\end{proof}

\begin{lemma}\label{lemma:JE2}
Let $r<1.$ Then there exists a $\delta>0$ so that for all holomorphic maps
$F:\overline{\Delta}\rightarrow \Omega$ and all $\epsilon>0$, there exists
a holomorphic map $G:\Delta(0,1+\delta) \rightarrow \Omega$ such that
$|F-G|<\epsilon$ for $|z|<r$ and $G-F=\mathcal O(z^2).$
\end{lemma}
\begin{proof}
We will use Lemma \ref{lemma:JE1} with the choices $R=\frac{1}{D}$ and $c=\frac{C}{D}.$ This will determine values $L$ and $\delta.$ Let $\epsilon>0.$ Assume now that $F:\overline{\Delta} \rightarrow \Omega.$ Fix any $n_0$ so that
$f^{n_0}\circ F(\Delta) \subset \mathbb B(1/2).$ It suffices to find a $G'$ which approximates $f^{n_0}\circ F,$ because then we can set $G=f^{-n_0}\circ G'.$ Hence we will assume that $F(\Delta)\subset \mathbb B(1/2).$

Next write
$f^{N}\circ F=(g_N,h_N).$ Then the functions $g'_N=\frac{2}{b^N} g_N,
h'_N=\frac{2}{b^N} h_N$ are maps of the unit disc into the unit disc. Let $\hat{g}'_N,\hat{h}'_N$
denote the sum of the terms up to degree $LN$ of the respective series for $g'_N,h'_N$ and also
$\tilde{g}'_N,\tilde{h}'_N$ the respective remainders. We use similar notation for $g_N,h_N.$
Then for all large enough $N$ we have that
$|\hat{g}'_N|,|\hat{h}'_N|<\frac{1}{D^N}$ if $|z|<1+\delta$ and
$|\tilde{g}'_N|<c^N\epsilon$ and $|\tilde{h}'_N|<c^N\epsilon$ if $|z|<r$. Hence $D^N \hat{g}'_N(z)$ and $D^N \hat{h}'_N(z)$
lie in the unit disc if $|z|<1+\delta.$ Therefore $\hat{g}_N:=\frac{D^N}{2}\hat{g'}_N$ and $\hat{h}_N:=\frac{D^N}{2}\hat{h'}_N$ map $\Delta(1+\delta)$ holomorphically into $\Delta(1/2)$.
It follows that
$$
f^{-N}(\hat{g}_N,\hat{h}_N)(\Delta(1+\delta))\subset \Omega.
$$

It remains to show that $f^{-N}\circ (\hat{g}_N,\hat{h}_N)$ approximates $F$ well on $\Delta(0,r).$ Notice that $|\tilde{g}'_N|,|\tilde{h}'_N|<c^N \epsilon$ if $|z|<r$. Hence
$$
\|(g'_N,h'_N)-(\hat{g}'_N,\hat{h}'_N)\|<2c^N\epsilon
$$
on $|z|<r$. Therefore
$$
\|(g_N,h_N)-(\frac{D^N}{2}\hat{g}'_N,\frac{D^N}{2}\hat{h}'_N)\|< D^Nc^N\epsilon
$$
on $|z|<r$.
Hence
$$
\|f^{N}\circ F-(\hat{g}_N,\hat{h}_N)\|<D^Nc^N\epsilon
$$
on $|z|<r$. Hence it follows from Lemma \ref{lemma:JE4} that
$$
\| F-f^{-N}(\hat{g}_N,\hat{h}_N)\|<\frac{D^Nc^N}{C^N}\epsilon=\epsilon
$$
on $|z|<r$.
\end{proof}

After inductive use of this lemma we obtain the following.

\begin{corollary}\label{cor:lines}
Let $p \in \Omega_{(f_n)}$ and $v \in T_p(\mathbb C^2)$. Then there exists a holomorphic map $\psi: \mathbb C \rightarrow \Omega_{(f_n)}$ with $\psi(0) = p$ and $\psi^\prime(p) = v$.
\end{corollary}

Applying the arguments in Lemmas \ref{lemma:JE1} and \ref{lemma:JE2} to maps of balls, we can even get the following.

\begin{corollary}\label{cor:image}
For any point $p \in \Omega_{(f_n)}$ there exists a holomorphic map $\Phi: \mathbb C^2 \rightarrow \Omega_{(f_n)}$, biholomorphic in a neighborhood of the the origin, such that $\Phi(0) = p$.
\end{corollary}

As a consequence we have obtained another proof of the following.

\begin{corollary}
There are no bounded nonconstant pluri-subharmonic functions on $\Omega_{(f_n)}$.
\end{corollary}
\begin{proof}
Suppose that there does exist a bounded, non-constant pluri-subharmonic function $u$. Then there exists a point $p \in \Omega_{(f_n)}$ such that $u$ is non-constant in any neighborhood of $p$. Let $\Phi$ be as in Corollary \ref{cor:image}, with $\Phi(0) = p$. Then $u \circ \Phi$ is a non-constant bounded pluri-subharmonic function on  $\mathbb C^2$, which is a contradiction.
\end{proof}

Recall that we have a natural exhaustion of $\Omega_{(f_n)}$ by relatively compact domains $U_n$, each biholomorphic to the unit ball, by writing
$$
U_n := (f_n \circ \cdots \circ f_0)^{-1}(\mathbb B).
$$
From the above proofs we see:

\begin{corollary}
For any given $n$ the map $\Phi$ can be made to arbitrarily well approximate the biholomorphic map from the unit ball to $U_n$.
\end{corollary}

\begin{problem}
How can one improve the map $\Phi?$
\begin{enumerate}
\item[(i)] Can we construct $\Phi$ with everywhere non-vanishing Jacobian determinant?
\item[(ii)] Can we make sure that $\Phi$ is injective?
\item[(iii)] Can we make sure that $\Phi(\mathbb C^2) = \Omega_{(f_n)}$?
\end{enumerate}
\end{problem}

Note that a positive answer to (ii) and (iii) would give a positive answer to Conjecture \ref{conj:main}. The same questions can be asked about holomorphic maps $\psi$ from $\mathbb C$ to $\Omega_{(f_n)}$:

\begin{problem} What can we say about $\psi$?
\begin{enumerate}
\item[(i)] Can we construct $\psi$ with constant rank $1$?
\item[(ii)] Can we make sure that $\psi$ is injective?
\item[(iii)] Can we make sure that $\psi$ is proper?
\end{enumerate}
\end{problem}

To end this section we discuss two models for Conjecture \ref{conj:stable}. In both cases the setting may well be explicit enough to obtain more control over the stable manifolds.

\begin{problem}
Let $d \ge 2$ and let $F$ be a holomorphic endomorphism of $\mathbb C^3$ of the form
$$
F: (z_1, z_2, z_3) \mapsto (z_1^d, G_{z_1}(z_2, z_3)).
$$
Suppose that for each $\lambda$ with $|\lambda| = 1$ the map $G_{\lambda}$ is an automorphism of $\mathbb C^2$ with an attracting fixed point at the origin. Let $p = (p_1, 0,0) \in \mathbb C^3$ with $|p_1| = 1$. Is the stable manifold $\Sigma_f^s(p)$ equivalent to $\mathbb C^2$?
\end{problem}

In this problem the translation from the the stable manifold $\Sigma_f^s(p)$ to the non-autonomous basin $\Omega_{(g_n)}$, where
$$
g_n = G_{p_1^{(d^n)}}
$$
is very concrete, and we have a better knowledge of how the maps $g_n$ vary with $n$. Can this knowledge be used to conclude anything about the non-autonomous basin of attraction? If so, it might be possible to prove Conjecture \ref{conj:stable} without proving Conjecture \ref{conj:main}.

Note that the skew product setting is not covered by Theorem \ref{thm:JV} of Jonsson and Varolin, or Theorem \ref{thm:Lyapunov} of Abate, Abbondalo and Majer. Consider for example the case where $d=2$, and assume that the eigenvalues of $DG_{z_1}(0)$ vary with $z_1$. Then we can find a periodic orbit in the $z_1$-circle on which the eigenvalues are not equal to the eigenvalues at the fixed point $z_1 = 1$. Now consider a $z_1$-orbit that remains very close to $1$ for a long time, then in a small number of steps gets close to the periodic orbit and remains close to it for very long time, then returns close to $1$ for an even longer time, and so on. We can conclude that there are no Lyapunov exponents.

\begin{problem}
Let $d \ge 2$ and let $F$ be a holomorphic endomorphism of $\mathbb C^3$ of the form
$$
F: (z_1, z_2, z_3) \mapsto (e^{2\pi i \theta} z_1, G_{z_1}(z_2, z_3)),
$$
with $\theta \in \mathbb R \setminus \mathbb Q$. Suppose that for each $\lambda$ with $|\lambda| = 1$ the map $G_{\lambda}$ is an automorphism of $\mathbb C^2$ with an attracting fixed point at the origin. Let $p = (p_1, 0,0) \in \mathbb C^3$ with $|p_1| = 1$. Is the stable manifold $\Sigma_f^s(p)$ equivalent to $\mathbb C^2$?
\end{problem}

\section{Proof of Theorem \ref{thm:main}}

Let us recall the statement of Theorem \ref{thm:main}.

\medskip

{\bf Theorem \ref{thm:main}.}
\emph{Let $(f_n)$ be a sequence of automorphisms of $\mathbb C^2$ which satisfies the conditions in Conjecture \ref{conj:main}, and suppose that the maps $f_n$ all have order of contact $k$. Assume further that the linear part of each map $f_n$ is diagonal. Then if $D^{k+1} < C$ the basin of attraction $\Omega_{(f_n)}$ is biholomorphic to $\mathbb C^2$.}

\medskip

The proof will be completed in three steps. In the first step time $[0, \infty)$ will be partitioned into intervals $I_j = [p_j, p_{j+1})$. We will refer to the intervals $I_j$ as \emph{trains}. On each train one of the coordinates will be contracted most rapidly by the derivatives $Df_n(0)$ \emph{on average}, but not for each map separately. In the second step we will find a non-autonomous conjugation on each train by linear maps so that in each train the new maps all contract the same coordinate most rapidly. This means that we can apply Lemma \ref{lemma:dominant2} below on each train separately. In the third and final step we worry about what happens at times $p_j$ where we switch from one train to the other. We refer to these three steps as respectively \emph{selecting}, \emph{directing}, and \emph{connecting} the trains.

Having directed the trains it will be easy to construct the maps $(g_n)$ and $(h_n)$ on each of the trains using the following Lemma, a special case of Lemma \ref{lemma:dominant}.

\begin{lemma}\label{lemma:dominant2}
Suppose that each map of the sequence $(f_n)$ has linear part of the form $(z, w) \mapsto (a_n z, b_n w)$, with $|b_n| \le |a_n|$. Then for any $k \ge 2$ we can find bounded sequences $(g_n)$ and $(h_n)$, with the maps $g_n$ lower triangular, such that Diagram \eqref{diagram} commutes up to jets of order $k$.
\end{lemma}

\subsection{Selecting the trains}

Here we will describe how to select the intervals $I_j$, and discuss the basic properties of these trains. The trains will be defined recursively depending on the linear parts of the sequence $(f_n)$. As before we write $(z,w) \mapsto (a_n z, b_n w)$ for the linear part of the map $f_n$. We then define
\begin{equation}
\begin{aligned}
f_{m,n} & = f_{m-1} \circ \cdots \circ f_n, \\
a_{m,n} & = a_{m-1} \cdot \cdots \cdot a_n, \; \; \mathrm{and} \\
b_{m,n} & = b_{m-1} \cdot \cdots \cdot b_n.
\end{aligned}
\end{equation}
Then let
\begin{equation}
\sigma_{m,n} = \log \left|\frac{a_{m,n}}{b_{m,n}}\right|.
\end{equation}
Our trains will depend only on the values of the function $\sigma$.

We let $\tilde{I_1} = [0,l)$, where $l$ is the smallest integer for which $\sigma_{l,0} \ge k$.  We define the trains recursively as follows. Suppose that we have already defined $\tilde{I_j} = [p,q)$. We consider all intervals $[r,s)$ satisfying $s \ge r \ge q$ and
\begin{equation}\label{eq:bound}
(-1)^j \cdot \sigma_{s,r} \ge k^{j+1}.
\end{equation}
We let $\tilde I_{j+1}$ be that interval $[r,s)$ for which $s$ is minimal and $(-1)^j\sigma_{s,r}$ is maximal, and we define $I_j = [p, r)$.

The intervals $I_j$ have some pleasant properties. First of all, from Equation \eqref{eq:bound} it is immediately clear that
\begin{equation*}
\sum \frac{|I_j|}{k^j} = \infty,
\end{equation*}
Hence we will be able to apply Lemma \ref{lemma:sparse}. Let us write $I_j = [p_j, p_{j+1})$, and $\tilde{I_j} = [p_j, q_j)$. Then for $n \in \tilde{I_j}$ we have
\begin{equation}\label{eq:train1}
(-1)^{j+1}\sigma_{n, p_j} \ge 0.
\end{equation}
By our choice of $\tilde{I_j}$ we also have
\begin{equation}\label{eq:train2}
(-1)^{j+1} \sigma_{q_j, n} \ge 0, \; \; \mathrm{and} \; \; k^j - \log(C) \ge (-1)^{j+1}\sigma_{q_j, p_j} \ge k^j.
\end{equation}
Furthermore, for $p_{j+1} \ge n \ge l \ge q_j$ we have
\begin{equation}\label{eq:train3}
(-1)^{j+1}\sigma_{n, l} > - k^{j+1} \; \; \mathrm{and} \; \; (-1)^{j+1} \sigma_{p_{j+1}, q_j} \ge 0.
\end{equation}

From now on we can forget the precise construction of the sequence $(I_j)$ and only work with the conditions in Equations \eqref{eq:train1}, \eqref{eq:train2} and \eqref{eq:train3}.

\subsection{Directing the trains}

We assume that the automorphisms $f_0, f_1, \ldots$ satisfy the conditions of Theorem \ref{thm:main}, and that we have constructed intervals $(I_j)$ satisfying Equations \eqref{eq:train1}, \eqref{eq:train2} and \eqref{eq:train3}. Here we change coordinates with a sequence of linear maps $l_n$ of the form
\begin{equation}
l_n(z,w) = (\theta_n z, \tau_n w).
\end{equation}
We write $\tilde{f}_n = l_{n+1} \circ f_n\circ  l_n^{-1}$ for the new maps, and we immediately see that $\tilde{f}_n$ is of the form
\begin{equation}
\tilde{f}_n (z,w) = (\tilde{a}_n z, \tilde{b}_n w) + O(k),
\end{equation}
where
\begin{equation}
\begin{aligned}
\tilde{a}_n & = \frac{\theta_{n+1}}{\theta_n} a_n, \; \; \mathrm{and}\\
\tilde{b}_n & = \frac{\tau_{n+1}}{\tau_n} b_n.
\end{aligned}
\end{equation}

Our goal is to define the maps $(l_n)$ in such a way that on each interval $I_j$ the maps $\tilde{f}_n$ satisfy the property $\tilde{a}_n \ge \tilde{b}_n$ (if $j$ is odd), and $\tilde{a}_n \le \tilde{b}_n$ (if $j$ is even). In order for the higher order terms of the maps $\tilde{f}_n$ not to blow up we also require that
\begin{equation}\label{boundeddistortion}
\theta_n^k \ge \tau_n, \; \; \mathrm{and} \; \theta_n \le \tau_n^k.
\end{equation}

In order to simplify the notation we let $j$ be odd, and write $\tilde{I}_j = [p,q)$ and $I_j = [p,r)$. The results are analogues for $j$ even. We assume that $\theta_p \ge e^{\frac{k}{k^2 - 1} k^j}$ and $\tau_p \ge e^{\frac{1}{k^2 - 1} k^j}$ and that the conditions in \eqref{boundeddistortion} are satisfied for $n = p$. Then for $n+1 \in I_j$ we recursively define
\begin{equation}
\theta_{n+1} = \left\{ \begin{aligned}
\theta_n \; \; \mathrm{if} \; |a_n| \ge |b_n|, \; \mathrm{or} \\
\frac{|b_n|}{|a_n|} \theta_n \; \; \mathrm{if} \; |a_n| < |b_n|.
\end{aligned}\right.
\end{equation}
Similarly we define
\begin{equation}
\tau_{n+1} = \left\{ \begin{aligned}
& \tau_n \; \; \; \; \; \; \; \; \mathrm{if} \; |b_n| \ge |a_n|, \; \mathrm{or} \\
& \min(\frac{|a_n|}{|b_n|} \tau_n, \theta_{n+1}^k) \; \; \mathrm{if} \; |b_n| < |a_n|.
\end{aligned}\right.
\end{equation}

\begin{lemma}
For $p \le n \le q$ we have
\begin{equation}
\tau_n^k \ge e^{k \cdot \sigma_{n, p}} \theta_n.
\end{equation}
\end{lemma}
\begin{proof}
Follows easily by induction on $n$.
\end{proof}

Note that $\tau_n \le \theta_n^k$ for all $n \in [p, r]$ follows immediately from the recursive definition of $\tau_{n+1}$ and the fact that $\theta_n$ can never decrease.

\begin{lemma}
For $p \le n \le r$ we have that $\theta_n \le \tau_n^k$.
\end{lemma}
\begin{proof}
For $n \le q$ this is guaranteed by the previous lemma. From $q$ on it follows by induction on $n$ that
\begin{equation}
\frac{\tau_n^k}{\theta_n} \ge e^{k^{j+1} + \sigma_{n,m}},
\end{equation}
for all $q \le m \le n$ for which $\sigma_{n,m} \le 0$. The statement in the Lemma now follows from \eqref{eq:train3}.
\end{proof}

Finally we need to check that $\theta_r$ and $\tau_r$ are large enough to satisfy the starting hypothesis for the next interval $I_{j+1}$.

\begin{lemma}
We have $\theta_r \ge e^{\frac{1}{k^2 - 1} k^{j+1}}$ and $\tau_p \ge e^{\frac{k}{k^2 - 1} k^{j+1}}$.
\end{lemma}
\begin{proof}
The estimate on $\theta_r$ is immediate since $\theta_n$ does not decrease with $n$. The estimate on $\tau_p$ follows from
\begin{equation}
\tau_p \ge \tau_q \ge e^{\sigma_{q, p}} \cdot e^{\frac{k}{k^2 - 1} k^j}.
\end{equation}
\end{proof}

Our conclusion is the following.

\begin{thm}\label{thm:directing}
Let $(\tilde{f}_n)$ be the sequence defined by
\begin{equation}
\tilde{f}_n = l_{n+1} \circ f_n \circ l_n^{-1}.
\end{equation}
Then $(\tilde{f}_n)$ is a bounded sequence of automorphisms, and $\Omega_{(\tilde{f}_n)} \cong \Omega_{(f_n)}$. Moreover
$$
(-1)^{n+1} \log(\frac{|\tilde{a}_n|}{|\tilde{b}_n|}) \ge 0
$$
for all $n \in I_j$.
\end{thm}
\begin{proof}
The conditions on the coefficients of the linear parts of the maps $\tilde{f}_n$ follows from the discussion earlier in this section. The fact that the sequence $(\tilde{f}_n)$ is bounded follows from the facts that the linear parts stay bounded, and that the higher order terms grow by at most a uniform constant.

To see that the two basins of attraction are biholomorphically equivalent (with biholomorphism $l_0$), note that $\tilde{f}_{n,0} = l_{n+1} \circ f_{n,0}\circ  l_0^{-1}$. Since the entries of the diagonal linear maps $l_{n+1}$ are always strictly greater than $1$, it follows that the basin of the sequence $(\tilde{f}_n)$ is contained in the $l_0$-image of the basin of the sequence $(f_n)$. The other direction follows from the fact that the coefficients of the maps $l_n$ grow strictly smaller than the rate at which orbits are contracted to the origin by the sequence $(f_n)$.
\end{proof}

\subsection{Connecting the trains}

The following result is a direct consequence of Lemma \ref{lemma:dominant}.

\begin{thm}\label{thm:triangulization}
Let $(f_n)$ be a bounded sequence of automorphisms of $\mathbb{C}^2$ whose linear parts are of the form $(z,w) \mapsto (a_n z, b_n w)$ and whose order of contact is $k$. Suppose that $|a_n| \ge |b_n|$ for all $n = 0, 1, \ldots$. Then there exists bounded sequences $(h_n)$ and $(g_n)$ such that $g_n \circ h_n = h_{n+1} \circ f_n + O(k+1)$ for all $n$. Here the maps $g_n$ can be chosen to be lower triangular maps, and the maps $h_n$ can be chosen of the form $(z,w) \mapsto \mathrm{Id} + h.o.t.$.
\end{thm}

In fact, even without the condition $|a_n| \ge |b_n|$ we can always change coordinates so that the maps $f_n$ become of the form
\begin{equation}
f_n(z,w) = (a_n z + c_n w^k, b_n w + d_n z^k) + O(k+1).
\end{equation}
Hence we may assume that our maps are all of this form. We outline the proof of Theorem \ref{thm:triangulization}. Our goal is to find sequences $(g_n)$ and $(h_n)$ so that the following diagram commutes up to degree $k$.
\begin{equation}
\begin{CD}
\mathbb{C}^2 @>f_0>> \mathbb{C}^2 @>f_1>> \mathbb{C}^2 @>f_2>> \cdots\\
@VVh_0V @VVh_1V @VVh_2V\\
\mathbb{C}^2 @>g_0>> \mathbb{C}^2 @>g_1>> \mathbb{C}^2 @>g_2>> \cdots
\end{CD}
\end{equation}
Here $h_n$ will be of the form
\begin{equation} \label{form:h_n}
h_n: (z,w) \mapsto (z + \alpha_n w^k, w),
\end{equation}
and $g_n$ will be of the form
\begin{equation}
g_n: (z,w) \mapsto (a_n z, b_n w + \gamma_n z^k).
\end{equation}
We need that
\begin{equation}
h_n = g_n^{-1} \circ h_{n+1} \circ f_n + O(k+1).
\end{equation}
It is clear that given the map $h_{n+1}$ we can always choose $g_n$ so that $h_n$ is of the form \eqref{form:h_n}, and that the map $g_n$ is unique. Hence we can view $\alpha_n$ as a function of $\alpha_{n+1}$. In fact, this function is affine and given by
\begin{equation}
\alpha_n = \frac{b_n^k}{a_n} \alpha_{n+1} + \frac{c_n}{a_n},
\end{equation}
or equivalently
\begin{equation}
\alpha_{n+1} = \frac{a_n}{b_n^k} \alpha_n + \frac{c_n}{b_n^k}.
\end{equation}
Note that both coefficients of these affine maps are uniformly bounded from above in norm, and that $|\frac{a_n}{b_n^k}| \ge \frac{1}{D}$. It follows that there exist a unique bounded orbit for this sequence of affine maps. In other words, we can choose a sequence $(h_n)$ whose $k$-th degree terms are uniformly bounded. It follows directly that the maps $g_n$ are also uniformly bounded, which completes the proof.

More strongly we note that we have much more flexibility if we have a sequence of intervals $(I_j)$ (with $j$ odd) and we want to find sequences $(h_n)$ and $(g_n)$ on each $I_j$ that are uniformly bounded over all $j$. In fact for each interval $I_j$ we can start with any value for $\alpha_{r_j - 1}$ and inductively define the maps $h_n$ backwards. We merely need to choose uniformly bounded starting values $\alpha_{r_j - 1}$.

The entire story works just the same for the intervals $I_j$ with $j$ even. In this case we have that $|a_n| \le |b_n|$, so we work with maps $h_n$ and $g_n$ of the form
\begin{equation}
h_n: (z,w) \mapsto (z , w + \beta_n z^k),
\end{equation}
and
\begin{equation}
g_n: (z,w) \mapsto (a_n z + \delta_n w^k, b_n w).
\end{equation}
The only matter to which we should pay attention is what happens where the different intervals are connected. Hence let us look at the following part of the commutative diagram
\begin{equation}
\begin{CD}
\cdots @>f_{r_j - 2}>> \mathbb{C}^2 @>f_{r_j - 1}>> \mathbb{C}^2 @>f_{r_j}>> \mathbb{C}^2 @>f_{r_j + 1}>> \cdots\\
@VV{}V @VVh_{r_j - 1}V @VVh_{r_j}V @VVh_{r_j + 1}V\\
\cdots @>g_{r_j - 2}>> \mathbb{C}^2 @>g_{r_j - 1}>> \mathbb{C}^2 @>g_{r_j}>> \mathbb{C}^2 @>g_{r_j + 1}>> \cdots
\end{CD}
\end{equation}
We consider the case where $j$ is odd and $j+1$ even, the other is similar. Imagine that we have constructed the maps $h_n$ and $g_n$ on the interval $I_{j+1}$. Then $h_{r_j}$ is of the form
\begin{equation}
h_{r_j}: (z,w) \mapsto (z , w + \beta_{r_j} z^k).
\end{equation}
As before we can find $g_{r_j -1}$ so that the map $h_{r_j - 1}$ given by $g_{r_j -1}^{-1} \circ h_{r_j} \circ f_{r_j -1}$ is of the form
\begin{equation}
h_{r_j -1}: (z,w) \mapsto (z + \alpha_{r_j -1}w^k, w).
\end{equation}
We note that while the map $g_{r_j -1}$ does depend on the coefficient $\beta_{r_j}$, the coefficient $\alpha_{r_j -1}$ (and equivalently also the map $h_{r_j -1}$) only depends on the coefficients of $f_{r_j-1}$ and not on those of $h_{r_j}$. Hence we can find uniformly bounded maps $h_n$ on each of the intervals, and therefore also uniformly bounded maps $g_n$, which are lower triangular for $n$ in each interval $[p_j, r_j)$ with $j$ odd, and upper triangular when $j$ is even. This completes the proof of Theorem \ref{thm:main}.

\end{document}